\documentclass[11pt,reqno]{amsart}
\usepackage[T1]{fontenc}
\usepackage{lmodern}
\usepackage{amsmath,amssymb,mathtools}
\usepackage[a4paper,margin=28mm]{geometry}
\usepackage{microtype}
\usepackage{enumitem}
\usepackage{needspace}
\usepackage[hidelinks]{hyperref}
\usepackage{tikz}
\newtheorem{theorem}{Theorem}[section]
\newtheorem{lemma}[theorem]{Lemma}
\newtheorem{proposition}[theorem]{Proposition}
\newtheorem{corollary}[theorem]{Corollary}
\theoremstyle{definition}
\newtheorem{definition}[theorem]{Definition}
\newtheorem{question}[theorem]{Question}
\newtheorem{example}[theorem]{Example}
\theoremstyle{remark}

\numberwithin{equation}{section}
\newcommand{\N}{\mathbb N}
\newcommand{\down}[1]{\mathord{\downarrow}#1}
\newcommand{\up}[1]{\mathord{\uparrow}#1}
\newcommand{\Sc}{\sigma}
\newcommand{\Sig}{\Sigma}
\newcommand{\mc}{\wedge}
\newcommand{\jc}{\vee}
\newcommand{\Idl}{\operatorname{Idl}}
\newcommand{\cl}{\operatorname{cl}}
\setlist{nosep}
\hypersetup{
 pdftitle={Scott topologies on meet-continuous domains},
 pdfauthor={Xiaoquan Xu},
 pdfsubject={Single-author working draft; Scott products, arbitrary-product criterion and conditional sobriety}}
\title[Scott topologies on meet-continuous domains]{Scott topologies on meet-continuous domains}
\author{Xiaoquan Xu}
\address{School of Computer Information Engineering, Nanchang Institute of Technology, Nanchang 330044, China}
\email{xiqxu2002@163.com}
\thanks{This research was supported by the National Natural Science Foundation of China (Nos.~12471070, 12071199).}
\begin{document}
\begin{abstract}
We study Scott products and sobriety of countable meet-continuous domains, meaning meet-continuous dcpos without any additional continuity or least-element assumption. Using the complete-lattice test-family theorem of Xu and Ji, we prove finite-product equality for those domains that are $L$-dcpos, and for the weaker class whose principal ideals have suprema of all nonempty subsets. For an arbitrary family of nonempty countable meet-continuous $L$-dcpos, we prove that the Scott topology on the order product equals the product of the factor Scott topologies if and only if only finitely many factors lack a least element. We also establish sobriety under bounded completeness and under additional common-upper-bound conditions. Assuming square-product equality, sobriety is characterized by Scott closedness of common-upper-bound sections associated with irreducible Scott-closed sets, with an equivalent sequential formulation in the countable case. Two extraction lemmas extend to meet-semilattice dcpos. The finite-product and sobriety questions for general countable meet-continuous domains remain unresolved here.
\end{abstract}
\maketitle

\section{Introduction and two questions}
The Scott topology relates order-theoretic directed convergence to topological continuity. Two basic issues are whether the Scott topology on an order product agrees with the ordinary product topology, and whether irreducible Scott-closed sets are principal ideals. These issues interact, but are distinct outside the setting of complete lattices; see \cite{Gierz,LawsonXu,LawsonXuLower}.

Xu and Ji \cite{XuJi} prove that every countable meet-continuous complete lattice has a sober Scott space and that arbitrary products of such lattices carry the product of their Scott topologies. A key step in that work is the construction, for each such lattice, of a countable family of directed sets that detects Scott openness. These complete-lattice results and the associated general test-family product criterion are used below as results of the companion paper, not as new claims of the present manuscript.

For general dcpos, meet-continuity does not require binary meets or joins. A dcpo $P$ is called a \emph{meet-continuous domain}, or equivalently a \emph{meet-continuous dcpo}, if, for every $x\in P$ and nonempty directed $D\subseteq P$,
\begin{equation}\label{eq:general-mc}
 x\le\bigvee D
 ~\Longrightarrow~
 x\in\cl_{\sigma(P)}(\down x\cap\down D),
\end{equation}
where $\Sig P$ denotes the Scott space; see \cite{KouLiuLuo} and \cite[Section~3]{JiaJungLi}. The terminology \emph{meet-continuous domain} is also used in \cite{LyuKou}. Throughout this paper the two expressions are synonymous: no way-below continuity, least element, or semilattice operation is implicit in the word \emph{domain}. We retain the term dcpo for arbitrary directed-complete posets. Condition~\eqref{eq:general-mc} agrees with the usual directed distributive law when binary meets exist. The complete-lattice theorem leads to the following two questions.

\begin{question}\label{quest:mc-dcpo-product}
Let $P$ and $Q$ be countable meet-continuous domains. Must
\[
 \Sc(P\times Q)=\Sc(P)\times\Sc(Q)
\]
hold?
\end{question}

\begin{question}\label{quest:mc-dcpo-sober}
Is the Scott space of every countable meet-continuous domain sober?
\end{question}

Neither question assumes a least element or any global semilattice operation. Throughout, $\Sc(P\times Q)$ refers to the Scott topology of the coordinatewise order, while $\Sc(P)\times\Sc(Q)$ denotes the ordinary topological product.

This manuscript records partial results toward these questions. Section~\ref{sec:local-tests} proves the finite-product equality when all principal ideals are complete lattices, that is, for $L$-dcpos in the terminology of Jia, Jung and Li \cite{JiaJungLi}. The local test families supplied by \cite{XuJi} can be combined because the ambient dcpo is countable. Section~\ref{sec:bottom} removes the local least-element requirement and proves sobriety for dcpos in which every nonempty upper-bounded subset has a global supremum.

For the general sobriety question, the absence of binary joins suggests replacing $\up(x\jc y)$ by $\up x\cap\up y$. The common-upper-bound method of Lawson and Xu \cite{LawsonXu,LawsonXuLower} yields sobriety from the strong $d$-space property together with the square product equality. Section~\ref{sec:upper-bounds} isolates the weaker condition needed only for irreducible Scott-closed sets and gives its sequential form. The distinction matters: a countable algebraic $L$-dcpo can fail to be a strong $d$-space. Section~\ref{sec:semilattice} identifies the parts of the meet-coordinate argument that continue to work when binary meets, but not binary joins, remain available.

Section~\ref{sec:arbitrary-products} gives an exact criterion for arbitrary products of countable meet-continuous $L$-dcpos: the Scott and ordinary product topologies agree exactly when all but finitely many factors have least elements. The necessity uses only the $L$-dcpo structure. The sufficiency combines the finite-product theorem with a finite-support argument that keeps every factor without a least element among the retained coordinates.

\section{Preliminaries and the complete-lattice input}
All posets and directed sets considered below are nonempty unless explicitly stated otherwise; countable sets may be finite. We write $\N=\{0,1,2,\ldots\}$. A subset $D$ of a poset is \emph{directed} if any two of its elements have a common upper bound in $D$. A \emph{dcpo} is a poset in which every directed set has a supremum. For $A\subseteq P$, put
\[
 \down A=\{x\in P:x\le a\text{ for some }a\in A\},~
 \up A=\{x\in P:a\le x\text{ for some }a\in A\}.
\]
We write $\down x$ and $\up x$ for singletons. An \emph{ideal} is a directed lower set. The set of ideals, ordered by inclusion, is denoted by $\Idl(P)$. An ideal is principal if it has the form $\down x$.

An upper set $U$ in a dcpo $P$ is \emph{Scott open} if $\bigvee D\in U$ implies $D\cap U\ne\emptyset$ for every directed $D\subseteq P$. All Scott open sets form the \emph{Scott topology} $\Sc(P)$, and the resulting space is denoted by $\Sig P$. Its closed sets are precisely the lower sets closed under directed suprema; their collection is $\Gamma(P)$. A nonempty closed set $A$ is \emph{irreducible} if any two open sets meeting $A$ have an intersection meeting $A$. We write $\operatorname{Irr}_c(\Sig P)$ for these sets. The space $\Sig P$ is \emph{sober} if each such $A$ equals $\down a$ for a unique $a\in P$.

A map between dcpos is Scott continuous if and only if it preserves directed suprema. Directed suprema in order products are computed coordinatewise. In particular, the coordinate projections and maps obtained by fixing one coordinate are Scott continuous. These facts concern Scott order products, not an assumed equality with ordinary topological products. For background, see \cite{AbramskyJung,Gierz}.

\begin{lemma}[{\cite[Lemma~1]{LawsonXuLower}}]\label{lem:cofinal-chain}
Every countable directed subset $D$ of a poset contains a countable cofinal chain $C$, so that $\down C=\down D$. If $D$ has no largest element, $C$ can be chosen strictly ascending. The two sets have the same supremum whenever it exists.
\end{lemma}
\begin{proof}
If $D$ has a largest element, use the corresponding singleton. Otherwise enumerate $D=\{d_n:n\in\N\}$ and set $c_0=d_0$. Given $c_n$, choose $e_n\in D$ with $e_n\nleq c_n$ and a common upper bound $c_{n+1}\in D$ of $c_n,d_{n+1},e_n$. Then $c_n<c_{n+1}$ and $d_i\le c_n$ for $i\le n$. Hence $C=\{c_n:n\in\N\}$ is cofinal in $D$, and $C,D$ have the same upper bounds.
\end{proof}

A complete lattice has suprema and infima of all subsets, including the empty subset. We shall also consider principal ideals having suprema only of nonempty subsets. These two conditions will be kept separate.

\begin{lemma}\label{lem:meet-mc}
If a dcpo $P$ has binary meets, condition \eqref{eq:general-mc} is equivalent to
\begin{equation}\label{eq:meet-mc}
 x\mc\bigvee D=\bigvee_{d\in D}(x\mc d)
 ~(x\in P,\ D\subseteq P\text{ directed}).
\end{equation}
\end{lemma}
\begin{proof}
Assume \eqref{eq:general-mc} and put $z=x\mc\bigvee D$. The set $\down z\cap\down D$ is generated downward by the directed set $\{z\mc d:d\in D\}$. Its Scott closure is the principal ideal of the latter set's supremum. Thus \eqref{eq:general-mc} gives
\[
 z=\bigvee_{d\in D}(z\mc d)=\bigvee_{d\in D}(x\mc d),
\]
since $d\le\bigvee D$. Conversely, if \eqref{eq:meet-mc} holds and $x\le\bigvee D$, the directed set $\{x\mc d:d\in D\}$ lies in $\down x\cap\down D$ and has supremum $x$. This implies \eqref{eq:general-mc}.
\end{proof}

\begin{definition}[{\cite[Section~4]{XuJi}}]\label{def:complete-tests}
A family $\mathcal D$ of directed subsets of a dcpo $P$ is \emph{complete for Scott openness} if an upper set $U\subseteq P$ is Scott open exactly when
\[
 \bigvee D\in U~\Longrightarrow~ D\cap U\ne\emptyset
 ~(D\in\mathcal D).
\]
\end{definition}

The next statement records the inputs from \cite[Theorem~1.1, Propositions~3.3 and~4.1]{XuJi}. In particular, the general product criterion does not require the factors themselves to be countable.

\begin{proposition}[Xu and Ji {\cite{XuJi}}]\label{prop:lattice-input}
\begin{enumerate}[label=\textup{(\arabic*)}]
\item Every countable meet-continuous complete lattice admits a countable family of directed subsets complete for Scott openness.
\item If dcpos $P$ and $Q$ each admit such a countable complete family, then $\Sc(P\times Q)=\Sc(P)\times\Sc(Q)$.
\item Every countable meet-continuous complete lattice has a sober Scott space.
\end{enumerate}
\end{proposition}

In \cite{XuJi}, the first assertion is proved by a meet-coordinate separation argument, extraction of directed punctured intervals, and join translation inside a maximal ideal avoiding a given upper set. The second assertion is a successive finite-selection argument, adapted from \cite[Lemma~4.1]{Miao}. The third follows by making binary join jointly continuous. We do not reproduce those complete-lattice and rectangle proofs here. The new step for $L$-dcpos is to localize the first assertion and combine the resulting families.

\section{\texorpdfstring{$L$}{L}-dcpos and local test families}\label{sec:local-tests}
\begin{definition}[{\cite[Definition~5.1]{JiaJungLi}}]\label{def:L-dcpo}
A dcpo $P$ is an \emph{$L$-dcpo} if every principal ideal $\down r$, $r\in P$, is a complete lattice in its induced order.
\end{definition}

No global least element is required in Definition~\ref{def:L-dcpo}. For example, an antichain with at least two elements is an $L$-dcpo but has no least element. Each principal ideal must nevertheless have its own least element. Thus the definition does not discard the empty supremum inside $\down r$. This terminology, from \cite{JiaJungLi}, extends the local lattice condition for $L$-domains discussed in \cite[Section~4.1]{AbramskyJung}, without assuming continuity. We therefore keep the name $L$-dcpo for this structural condition, even when the ambient dcpo is called a meet-continuous domain.

For $x,y\le r$, write $x\jc_r y$ and $x\mc_r y$ for their join and meet in $\down r$. The join may depend on $r$ and need not be a join in $P$. The meet is a meet in $P$, since every common lower bound of $x,y$ lies in $\down r$. Pairs without a common upper bound are not required to have a meet.

\begin{lemma}\label{lem:principal-localization}
Let $P$ be a dcpo and $r\in P$.
\begin{enumerate}[label=\textup{(\arabic*)}]
\item The principal ideal $\down r$ is a sub-dcpo, and its Scott topology is the subspace topology induced by $\Sig P$.
\item $P$ is meet-continuous if and only if each $\down r$ is meet-continuous.
\end{enumerate}
\end{lemma}
\begin{proof}
Any directed subset of $\down r$ has its supremum in $\down r$, computed as in $P$. If $A$ is Scott closed in $\down r$, it is lower in $P$, because $\down r$ is lower, and it is closed under directed suprema in $P$. Thus $A$ is Scott closed in $P$. Conversely, intersections of Scott-closed subsets of $P$ with $\down r$ are Scott closed in $\down r$. This proves~(1).

Suppose $P$ is meet-continuous, $D\subseteq \down r$ is directed and $x\in \down r$ satisfies $x\le\bigvee D$. The set $\down x\cap\down D$ is contained in $\down r$, and its downsets computed in $P$ and in $\down r$ agree. Apply \eqref{eq:general-mc} and the subspace equality in~(1) to obtain meet-continuity of $\down r$. Conversely, for $x\le\bigvee D$ in $P$, set $r=\bigvee D$ and apply meet-continuity in $\down r$; its inclusion into $P$ is continuous, so \eqref{eq:general-mc} follows.
\end{proof}

\begin{proposition}\label{prop:local-tests}
Let $P$ be a countable dcpo. Suppose that each $\down r$, $r\in P$, admits a countable family $\mathcal D_r$ of directed sets complete for its Scott openness. Then
\begin{equation}\label{eq:local-test-union}
 \mathcal D_P=\bigcup_{r\in P}\mathcal D_r
\end{equation}
is a countable family of directed subsets of $P$ complete for Scott openness.
\end{proposition}
\begin{proof}
The family is countable, and each of its members has the same directed supremum in its principal ideal and in $P$. Necessity of the tests is immediate. Suppose that an upper set $U\subseteq P$ passes every test in $\mathcal D_P$. For each $r$, the upper set $U\cap\down r$ passes every test in $\mathcal D_r$, and hence is Scott open in $\down r$. Given a directed $D\subseteq P$ with $r=\bigvee D\in U$, work inside $\down r$ to obtain $D\cap U\ne\emptyset$. Thus $U$ is Scott open in $P$.
\end{proof}

\begin{theorem}\label{thm:L-dcpo-products}
Every countable meet-continuous $L$-dcpo admits a countable family of directed subsets complete for Scott openness. Consequently, for every finite family $P_1,\ldots,P_n$ of countable meet-continuous $L$-dcpos,
\begin{equation}\label{eq:L-dcpo-products}
 \Sc\!\left(\prod_{i=1}^{n}P_i\right)
 =\prod_{i=1}^{n}\Sc(P_i).
\end{equation}
No global least elements are assumed.
\end{theorem}
\begin{proof}
For a countable meet-continuous $L$-dcpo $P$, Lemma~\ref{lem:principal-localization} makes every principal ideal $\down r$ a countable meet-continuous complete lattice. Proposition~\ref{prop:lattice-input}\textup{(1)} supplies a countable complete test family on each $\down r$, and Proposition~\ref{prop:local-tests} combines them into one for $P$. Proposition~\ref{prop:lattice-input}\textup{(2)} now proves the binary product equality.

Finite products are again countable $L$-dcpos, since
\[
 \down(r_1,\ldots,r_n)=\prod_{i=1}^{n}\down r_i.
\]
Each principal ideal on the right is a meet-continuous complete lattice, with operations and directed suprema computed coordinatewise. Lemma~\ref{lem:principal-localization} therefore gives meet-continuity of the product. Induction proves \eqref{eq:L-dcpo-products}; the empty product is a singleton.
\end{proof}

The local tests may be described explicitly using the construction of \cite[Proposition~3.3]{XuJi}. For $\down r$, put
\[
 \mathcal Q_{\down r}=\{(t,d)\in \down r\times \down r : t\nleq d,\ [t\mc_r d,t)
     \text{ is directed and }\bigvee[t\mc_r d,t)=t\}.
\]
The resulting tests are indexed by triples $(r,t,d)$ with $(t,d)\in\mathcal Q_{\down r}$, and are
\[
 D^{\,r}_{t,d}=\{x\jc_r d:t\mc_r d\le x<t\}.
\]
All lattice operations here take place inside $\down r$. Thus no global binary join is being used. Theorem~\ref{thm:L-dcpo-products} concerns finite products. The exact extent to which it extends to arbitrary products is determined in Theorem~\ref{thm:arbitrary-L-products}; Example~\ref{ex:antichain-infinite-product} records the two-point-antichain obstruction already noted in \cite[Section~5]{XuJi}.

\section{Least elements and bounded completeness}\label{sec:bottom}
One can weaken the local completeness assumption further, but it is useful to keep this weakening distinct from the established definition of an $L$-dcpo.

\begin{lemma}\label{lem:adjoin-bottom}
Let $P_\bot=P\cup\{\bot\}$ be obtained from a dcpo $P$ by adjoining a new least element. Then $P_\bot$ is a dcpo,
\begin{equation}\label{eq:bottom-scott}
 \Sc(P_\bot)=\Sc(P)\cup\{P_\bot\},
\end{equation}
where subsets of $P$ are viewed as subsets of $P_\bot$. Moreover, $P$ is meet-continuous if and only if $P_\bot$ is meet-continuous, and $\Sig P$ is sober if and only if $\Sig P_\bot$ is sober.
\end{lemma}
\begin{proof}
If a directed set in $P_\bot$ has a point other than $\bot$, deleting $\bot$ leaves a nonempty directed set in $P$ with the same supremum. This proves directed completeness and \eqref{eq:bottom-scott}. For meet-continuity at an old point $x\in P$, the same deletion and the Scott-open neighbourhoods of $x$ reduce \eqref{eq:general-mc} to the corresponding condition in $P$. At $\bot$, the condition is automatic, since $\bot\in\down\bot\cap\down D$ for every nonempty directed $D$.

Every nonempty Scott-closed subset of $P_\bot$ has the form $\{\bot\}\cup A$, where $A$ is Scott closed in $P$. If $A\ne\emptyset$, this set is irreducible exactly when $A$ is irreducible, by \eqref{eq:bottom-scott}. Its being a principal ideal is likewise equivalent to $A$ being principal in $P$. The remaining nonempty case is $\{\bot\}$, itself principal. This proves the assertion about sobriety.
\end{proof}

\begin{corollary}\label{cor:bottom-free-local-products}
Let $P_1,\ldots,P_n$ be countable meet-continuous domains. Suppose that, for each $i$ and $r\in P_i$, every nonempty subset of $\down r$ has a supremum in $\down r$. Then
\[
 \Sc\!\left(\prod_{i=1}^{n}P_i\right)
 =\prod_{i=1}^{n}\Sc(P_i).
\]
Neither global least elements nor least elements of the principal ideals are required.
\end{corollary}
\begin{proof}
The stated local condition is equivalent to $(P_i)_\bot$ being an $L$-dcpo. Indeed, for an old point $r$, the principal ideal in $(P_i)_\bot$ is $\{\bot\}\cup\down_{P_i}r$. The assumed nonempty suprema, together with the new bottom, provide suprema of all its subsets, so it is a complete lattice. Conversely, the supremum of a nonempty subset of $\down_{P_i}r$ in this lattice cannot be the new bottom, and is therefore a supremum in $\down_{P_i}r$.

By Lemma~\ref{lem:adjoin-bottom}, the lifted dcpos are countable, meet-continuous $L$-dcpos. Apply Theorem~\ref{thm:L-dcpo-products} to their finite product. Each $P_i$ is Scott open in $(P_i)_\bot$, and $\prod_iP_i$ is Scott open in $\prod_i(P_i)_\bot$. On a Scott-open subset $O$ of a dcpo, the intrinsic Scott topology is the subspace topology: a directed set with supremum in $O$ has a cofinal tail in $O$, which tests openness inside $O$. Restricting the product equality for the lifted dcpos therefore gives the desired equality.
\end{proof}

In particular, the corollary applies if each principal ideal is merely a lattice, without a least element: it is already a dcpo, so joins of nonempty finite subsets form a directed set whose supremum is the join of the entire nonempty subset. For example, the nonpositive integers with their usual order satisfy this weaker hypothesis but are not an $L$-dcpo as defined above.

The next hypothesis concerns a supremum in the whole poset, not merely a join computed below one selected upper bound. It is therefore stronger than the local condition in Corollary~\ref{cor:bottom-free-local-products}.

\begin{proposition}\label{prop:bounded-complete-sober}
Let $P$ be a countable meet-continuous domain such that every nonempty upper-bounded subset of $P$ has a supremum in $P$. Then $\Sig P$ is sober. Finite products of such dcpos also satisfy the Scott product equality. No least element of $P$ is assumed.
\end{proposition}
\begin{proof}
The product assertion follows from Corollary~\ref{cor:bottom-free-local-products}. For sobriety, let $B=P_\bot$. Then $B$ is a countable meet-continuous domain in which every upper-bounded subset, including the empty subset, has a supremum. Adjoin a new greatest element $\top$ to obtain $K=B\cup\{\top\}$. Every subset of $B$ bounded in $B$ has its old supremum, and every subset unbounded in $B$ has supremum $\top$. Thus $K$ is a complete lattice.

For $x,y\in B$, their meet exists in $B$: it is the supremum of their nonempty, upper-bounded set of common lower bounds. Hence $B$ satisfies the meet-continuity identity \eqref{eq:meet-mc}. These binary meets do not change on adjoining $\top$. A directed subset of $K$ not containing $\top$ has its supremum in $B$, whereas a directed subset containing $\top$ has a greatest member. These two cases, and $\top\mc x=x$, verify \eqref{eq:meet-mc} for $K$. Therefore $K$ is a countable meet-continuous complete lattice, so $\Sig K$ is sober by Proposition~\ref{prop:lattice-input}\textup{(3)}.

The subset $B$ is Scott closed in $K$, and its intrinsic Scott topology is the subspace topology. Closed subspaces of sober spaces are sober: an irreducible closed subset of a closed subspace is also irreducible and closed in the ambient space, and its generic point belongs to that subset. Thus $\Sig B$ is sober. Lemma~\ref{lem:adjoin-bottom} now gives sobriety of $\Sig P$.
\end{proof}

\section{Common upper bounds and sobriety}\label{sec:upper-bounds}
For $A\in\Gamma(P)$ and $x\in P$, put
\begin{equation}\label{eq:upper-bound-sets}
 \begin{aligned}
 D_{A,x}&=\down(\up x\cap A),\\
 E_A&=\bigcup_{a\in A}(\down a\times\down a).
 \end{aligned}
\end{equation}
Thus $y\in D_{A,x}$ if and only if $(x,y)\in E_A$, or equivalently $x$ and $y$ have a common upper bound in $A$. Both sets are lower, and $E_A$ is symmetric. The complementary upper set is
\[
 (P\times P)\setminus E_A
 =\{(x,y):\up x\cap\up y\subseteq P\setminus A\}.
\]
This is the common-upper-bound replacement for the inverse image of $P\setminus A$ under binary join. The set $\up x\cap\up y$ need not be nonempty or Scott compact, so no such assumption is made on the values of this correspondence.

Recall that a $T_0$ space $X$ is a \emph{strong $d$-space} if, for every nonempty directed set $D$ in its specialization order, $x\in X$ and open $U\subseteq X$,
\[
 \bigcap_{d\in D}(\up d\cap\up x)\subseteq U
 ~\Longrightarrow~
 \up d\cap\up x\subseteq U\text{ for some }d\in D.
\]
For Scott spaces, this condition is equivalent to $D_{A,x}\in\Gamma(P)$ for every $A\in\Gamma(P)$ and $x\in P$ \cite[Theorem~61]{LawsonXuLower}. Together with the square product equality it implies sobriety \cite[Theorem~5.11]{LawsonXu}, \cite[Theorem~76]{LawsonXuLower}. For this last implication, however, only irreducible closed sets need to be considered.

\begin{lemma}\label{lem:upper-bound-sections}
For a dcpo $P$ and $A\in\Gamma(P)$, the following two conditions are equivalent:
\begin{enumerate}[label=\textup{(\arabic*)}]
\item $D_{A,x}\in\Gamma(P)$ for every $x\in P$.
\item $E_A$ is Scott closed in $P\times P$.
\end{enumerate}
\end{lemma}
\begin{proof}
Suppose~(1), and let $D\subseteq E_A$ be directed with $\bigvee D=(p,q)$. Fix $(u,v)\in D$. The tail of $D$ above $(u,v)$ is cofinal in $D$, and all its first coordinates belong to $B_{A,v}$. Their supremum is $p$, so $p\in B_{A,v}$. By symmetry, $v\in B_{A,p}$. This holds for every second coordinate of $D$. Since $B_{A,p}$ is Scott closed, $q\in B_{A,p}$, proving $(p,q)\in E_A$. Lowerness of $E_A$ now proves~(2).

Conversely, for fixed $x$, the map $y\mapsto(x,y)$ preserves directed suprema and has inverse image $D_{A,x}$ of $E_A$, so~(2) implies~(1).
\end{proof}

\begin{proposition}\label{prop:irreducible-upper-bound-sober}
Let $P$ be a dcpo such that
\begin{equation}\label{eq:square-product-hypothesis}
 \Sc(P\times P)=\Sc(P)\times\Sc(P).
\end{equation}
Then $\Sig P$ is sober if and only if
\begin{equation}\label{eq:irreducible-upper-bound-condition}
 \down(\up x\cap A)\in\Gamma(P)
 ~\bigl(A\in\operatorname{Irr}_c(\Sig P),\ x\in P\bigr).
\end{equation}
\end{proposition}
\begin{proof}
If $\Sig P$ is sober, then $A=\down a$ for some $a\in P$, and
\[
 D_{\down a,x}=
 \begin{cases}
  \down a,&x\le a,\\
  \emptyset,&x\nleq a.
 \end{cases}
\]
Thus \eqref{eq:irreducible-upper-bound-condition} holds, even without \eqref{eq:square-product-hypothesis}.

Conversely, fix $A\in\operatorname{Irr}_c(\Sig P)$. By Lemma~\ref{lem:upper-bound-sections}, $E_A$ is Scott closed in $P\times P$. By \eqref{eq:square-product-hypothesis}, it is closed in the ordinary product topology. It contains the diagonal $\Delta_A=\{(a,a):a\in A\}$. For $x,y\in A$, every open rectangle about $(x,y)$ meets $\Delta_A$, because any two open sets meeting $A$ have an intersection meeting $A$. Hence
\[
 A\times A\subseteq\cl_{\sigma(P)\times\sigma(P)}\Delta_A\subseteq E_A.
\]
Consequently, $A$ is directed. Scott closedness gives $A=\down\bigvee A$, proving sobriety.
\end{proof}

\begin{corollary}\label{cor:sequential-upper-bound}
For a countable dcpo $P$ satisfying \eqref{eq:square-product-hypothesis}, $\Sig P$ is sober if and only if the following holds: whenever $A\in\operatorname{Irr}_c(\Sig P)$, $x\in A$, and $y_0\le y_1\le\cdots$ has supremum $y$, then
\begin{equation}\label{eq:sequential-upper-bound-lifting}
 \begin{aligned}
 &\bigl(\forall n\in\N\ \exists a_n\in A:\ x\le a_n,
                  \ y_n\le a_n\bigr)\\
 &\hspace{20mm}\Longrightarrow~
 \exists a\in A:\ x\le a,\ y\le a.
 \end{aligned}
\end{equation}
\end{corollary}
\begin{proof}
For $x\notin A$, the set $D_{A,x}$ is empty because $A$ is lower. For $x\in A$, condition \eqref{eq:sequential-upper-bound-lifting} says exactly that the lower set $D_{A,x}$ is closed under suprema of nondecreasing sequences. By countability and Lemma~\ref{lem:cofinal-chain}, this is equivalent to its being Scott closed. Apply Proposition~\ref{prop:irreducible-upper-bound-sober}.
\end{proof}

The witnesses $a_n$ in \eqref{eq:sequential-upper-bound-lifting} need not form a directed set. In a lattice they can be replaced by the ascending sequence $x\jc y_n\le a_n$, whose supremum remains in $A$. Without binary joins, this replacement is unavailable. We have not established \eqref{eq:sequential-upper-bound-lifting} for all countable meet-continuous domains; it is a reduction of Question~\ref{quest:mc-dcpo-sober}, conditional on the square product equality, not an affirmative solution to it.

\begin{example}[{\cite[Example~58]{LawsonXuLower}}]\label{ex:Ldcpo-not-strong}
Countability and meet-continuity do not imply the strong $d$-space condition. Let
\[
 P=\{a_n:n\ge1\}\cup\{b,\omega_0\}\cup\{\omega_n:n\ge1\},
\]
with the order (see Figure 1) generated by
\[
 a_1<a_2<\cdots<\omega_0,
 ~ b<\omega_n,
 ~ a_m<\omega_n~(m\le n).
\]

\vspace{0.1cm}

\begin{center}
\begin{tikzpicture}[x=0.7cm,y=0.7cm,
  dot/.style={circle,fill=black,inner sep=1.8pt},
  every node/.style={inner sep=1pt}]
  \node[dot,label=left:$a_1$] (a1) at (0,0) {};
  \node[dot,label=left:$a_2$] (a2) at (0,1.2) {};
  \node[dot,label=left:$a_3$] (a3) at (0,2.5) {};
  \node[dot] (c1) at (0,3.25) {};
  \node[dot] (c2) at (0,3.75) {};
  \node[dot] (c3) at (0,4.25) {};
  \node[dot,label=left:$a_n$] (an) at (0,4.9) {};
  \node[dot] (c4) at (0,5.45) {};
  \node[dot] (c5) at (0,5.95) {};
  \node[dot] (c6) at (0,6.45) {};
  \node[dot,label=left:$\omega_0$] (w0) at (0,7.05) {};

  \node[dot,label=above:$\omega_1$] (w1) at (2.1,7.05) {};
  \node[dot,label=above:$\omega_2$] (w2) at (4.2,7.05) {};
  \node[dot,label=above:$\omega_3$] (w3) at (6.3,7.05) {};
  \node[dot] (td1) at (7.2,7.05) {};
  \node[dot] (td2) at (7.8,7.05) {};
  \node[dot] (td3) at (8.4,7.05) {};
  \node[dot,label=above:$\omega_n$] (wn) at (10.0,7.05) {};
  \node[dot] (td4) at (11.0,7.05) {};
  \node[dot] (td5) at (11.6,7.05) {};
  \node[dot] (td6) at (12.2,7.05) {};
  \node[dot,label=below:$b$] (b) at (6.1,2.2) {};

  \draw (a1)--(a2)--(a3)--(c1)--(c2)--(c3)--(an)--(c4)--(c5)--(c6)--(w0);

  \draw (b)--(w1);
  \draw (b)--(w2);
  \draw (b)--(w3);
  \draw (b)--(wn);

  \draw (a1)--(w1);
  \draw (a2)--(w2);
  \draw (a3)--(w3);
  \draw (an)--(wn);
\end{tikzpicture}

\smallskip
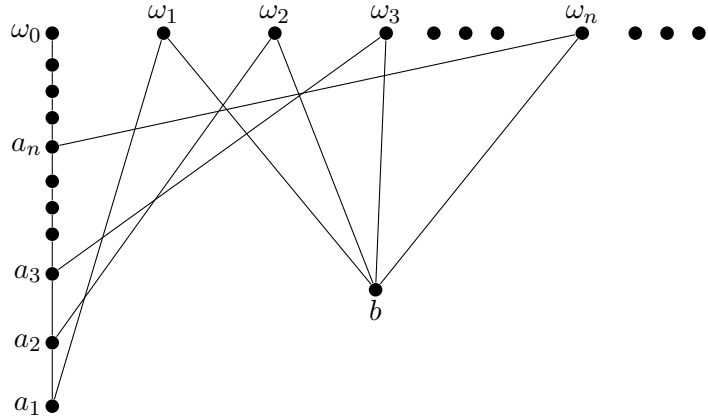
\small Figure 1. The poset $P$ in Example~\ref{ex:Ldcpo-not-strong}.
\end{center}
A directed subset containing a maximal point has that point as its largest element. A directed subset without a largest element cannot contain $b$ together with any $a_n$, because their common upper bounds are maximal. It must therefore be an infinite subset of the ascending chain $\{a_n:n\ge1\}$, and its supremum is $\omega_0$. This also shows that all points other than $\omega_0$ are compact. Since $\omega_0=\bigvee_n a_n$, the dcpo $P$ is countable and algebraic, hence meet-continuous. Nevertheless,
\[
 \up\omega_0\cap\up b=\emptyset,
 ~
 \up a_n\cap\up b=\{\omega_m:m\ge n\}\ne\emptyset.
\]
The lower set $B_{P,b}$ contains every $a_n$ but not their supremum $\omega_0$, so it is not Scott closed.

The same obstruction persists after adjoining a new bottom. Moreover, $(P)_\bot$ is an $L$-dcpo: each principal ideal is either a complete chain or a finite lattice. Hence even a pointed countable algebraic $L$-dcpo need not have a strong $d$-space as its Scott space. These examples are sober, as they are algebraic, and are not counterexamples to either question above. Their role is to exclude the stronger intermediate assertion for all Scott-closed sets.
\end{example}

The following consequences combine the proved local product result with the established sobriety criteria of Lawson and Xu \cite{LawsonXu,LawsonXuLower}; they do not require an affirmative answer to Question~\ref{quest:mc-dcpo-product} in general. We use property $R$, well-filteredness, and coherence as in \cite[Sections~2 and~3]{LawsonXuLower}. Here $\lambda(P)$ is the Lawson topology, the join of the Scott topology and the lower topology generated by the sets $P\setminus\up x$. The space $(P,\lambda(P))$ is \emph{upper-semicompact} if every principal upper set is Lawson compact.

\Needspace{10\baselineskip}
\begin{corollary}\label{cor:conditional-sobriety}
Let $P$ be a countable meet-continuous domain such that every nonempty subset of every principal ideal has a supremum in that ideal. In particular, $P$ may be an $L$-dcpo. If any one of the following conditions holds, then $\Sig P$ is sober:
\begin{enumerate}[label=\textup{(\arabic*)}]
\item $\Sig P$ is a strong $d$-space;
\item $\Sig P$ has property $R$;
\item $\Sig P$ is well-filtered and coherent;
\item $(P,\lambda(P))$ is upper-semicompact.
\end{enumerate}
\end{corollary}
\begin{proof}
Corollary~\ref{cor:bottom-free-local-products} gives the square-product equality~\eqref{eq:square-product-hypothesis}. Under~(1), all the sets $D_{A,x}$ are Scott closed by \cite[Theorem~61]{LawsonXuLower}, so Proposition~\ref{prop:irreducible-upper-bound-sober} applies. Conditions~(2)--(4) each imply~(1), by \cite[Lemma~50\textup{(3)}, Proposition~33\textup{(2)}, and Proposition~42]{LawsonXuLower}; equivalently one may apply \cite[Theorem~76 and Corollary~77]{LawsonXuLower} directly.
\end{proof}

\begin{corollary}\label{cor:finite-upper-bound-generators}
Let $P$ satisfy the local hypotheses of Corollary~\ref{cor:conditional-sobriety}. Suppose that for all $x,y\in P$ there is a finite set $F_{x,y}\subseteq P$, possibly empty, such that
\[
 \up x\cap\up y=\up F_{x,y}.
\]
Then $\Sig P$ is sober.
\end{corollary}
\begin{proof}
The Scott space of a dcpo is a $d$-space. The finite-generation assumption implies property $R$ by \cite[Proposition~33\textup{(1)}]{LawsonXuLower}. Apply Corollary~\ref{cor:conditional-sobriety}\textup{(2)}.
\end{proof}

\section{Extraction in meet-semilattice dcpos}\label{sec:semilattice}
The extraction part of \cite{XuJi} extends beyond complete lattices when binary meets remain available. This section makes the dependence on these operations explicit. For $s<t$ in a poset, write
\[
 [s,t)=\{x:s\le x<t\}.
\]
Such an interval is not assumed to be directed.

\begin{lemma}\label{lem:semilattice-isolation}
Let $P$ be a countable meet-continuous domain with binary meets, and let $(p_n)$ be an acsending sequence with supremum $r$. For every nonempty $S\subseteq\down r$, there exist $t\in S$ and $m\in\N$ such that
\begin{equation}\label{eq:semilattice-isolate}
 \{y\in S:y\mc p_m=t\mc p_m\}=\{t\}.
\end{equation}
\end{lemma}
\begin{proof}
We use the meet-coordinate argument of \cite[Lemma~3.1]{XuJi}. By Lemma~\ref{lem:meet-mc},
\[
 x=\bigvee_n(x\mc p_n)~(x\le r).
\]
Distinct elements of $\down r$ therefore differ at some coordinate, and a difference at one coordinate persists at all later coordinates: equality after meeting with $p_m$ implies equality after meeting with any $p_k\le p_m$.

Suppose \eqref{eq:semilattice-isolate} fails and enumerate all of $\down r$ as $\{e_n:n\in\N\}$, allowing repetitions. Choose $w\in S$, put $m_0=0$ and $b_0=w\mc p_0$. Recursively maintain a nonempty fibre
\[
 F_n=\{y\in S:y\mc p_{m_n}=b_n\}.
\]
It contains at least two points, so choose $y_n\in F_n\setminus\{e_n\}$ and $m_{n+1}>m_n$ with
\[
 b_{n+1}:=y_n\mc p_{m_{n+1}}\ne e_n\mc p_{m_{n+1}}.
\]
Then $F_{n+1}$ contains $y_n$, and
\[
 b_{n+1}\mc p_{m_n}=b_n,~ b_n\le b_{n+1}.
\]
Consequently $b=\bigvee_n b_n$ exists and belongs to $\down r$. For every fixed $n$, repeated compatibility gives $b_k\mc p_{m_n}=b_n$ for $k\ge n$, while the earlier terms are at most $b_n$. Equation~\eqref{eq:meet-mc} implies
\[
 b\mc p_{m_n}=\bigvee_k(b_k\mc p_{m_n})=b_n.
\]
Thus $b$ differs from $e_n$ at coordinate $m_{n+1}$ for every $n$, contradicting the enumeration of $\down r$. No binary join or nondirected supremum has been used.
\end{proof}

\begin{proposition}\label{prop:semilattice-puncture}
Let $P$ be a countable meet-continuous domain with binary meets. For each nonprincipal ideal $J$ of $P$, there exist $s<t\le\bigvee J$ such that $[s,t)$ is directed and
\begin{equation}\label{eq:semilattice-puncture}
 \down[s,t)\subseteq J,~ \bigvee[s,t)=t\notin J.
\end{equation}
\end{proposition}
\begin{proof}
The coordinate argument is that of \cite[Lemma~3.2]{XuJi}; the directedness argument below replaces the use of binary joins. Set $r=\bigvee J$. By Lemma~\ref{lem:cofinal-chain}, choose a strictly ascending cofinal sequence $(p_n)$ in $J$. Since $J$ is nonprincipal, $r\notin J$. Apply Lemma~\ref{lem:semilattice-isolation} to $S=\down r\setminus J$, and let $t,m$ be as there. Put $s=t\mc p_m$. Then $s\in J$ and $s<t$.

If $s\le x\le t$, then $x\mc p_m=s$. Equation~\eqref{eq:semilattice-isolate} forces every such $x<t$ to belong to $J$. Thus $[s,t)\subseteq J$.

For $x,y\in[s,t)$ choose $j\in J$ with $x,y\le j$. Then
\[
 x,y\le t\mc j\in J,~ s\le t\mc j<t,
\]
since $t\notin J$. This supplies a common upper bound in $[s,t)$ and proves directedness without a binary join. Finally, for $n\ge m$, the elements $t\mc p_n$ lie in $[s,t)$ and have supremum $t$ by \eqref{eq:meet-mc}. Lowerness of $J$ completes \eqref{eq:semilattice-puncture}.
\end{proof}

The endpoint $t$ need not equal $\bigvee J$. The two extraction statements above do not yet produce a countable complete family for an arbitrary meet-semilattice dcpo. In the complete-lattice proof, the later maximal-ideal extension and join translation require binary joins; neither operation has been replaced here. In an $L$-dcpo these operations can instead be performed within a principal ideal, as in Section~\ref{sec:local-tests}.

\section{Arbitrary products and least elements}\label{sec:arbitrary-products}
The finite-product theorem extends to arbitrary products precisely when only finitely many factors lack least elements. We first separate the structural obstruction from the finite-support argument. All factors in this section are nonempty, and arbitrary products are taken in the usual set-theoretic setting with choice.

\begin{lemma}\label{lem:minimal-component-decomposition}
Let $P$ be an $L$-dcpo, and let $\operatorname{Min}(P)$ denote its set of minimal elements. For $x\in P$, put $b(x)=\min(\down x)$. Then:
\begin{enumerate}[label=\textup{(\arabic*)}]
\item $b(x)$ is the unique minimal element of $P$ below $x$, and $x\le y$ implies $b(x)=b(y)$;
\item the sets $P_m=\up m$, $m\in\operatorname{Min}(P)$, form a partition of $P$ into Scott-open and Scott-closed sub-dcpos, with no comparable points in distinct parts;
\item $P$ has a least element if and only if $\operatorname{Min}(P)$ is a singleton.
\end{enumerate}
No countability or meet-continuity assumption is needed.
\end{lemma}
\begin{proof}
The element $b(x)$ exists because $\down x$ is a complete lattice. If $z\le b(x)$, then $z\in\down x$, so $b(x)\le z$ and $z=b(x)$. Thus $b(x)$ is minimal in $P$. If $m$ is any minimal element below $x$, then $b(x)\le m$ and hence $b(x)=m$.

If $x\le y$, then $b(y)\le x$, so $b(x)\le b(y)$. Also $b(x)\le y$, so $b(y)\le b(x)$. This proves~(1). Every point lies in exactly one $P_m$, and~(1) excludes comparabilities between different parts. Each $P_m$ is upper by definition. It is also lower: if $z\le x\in P_m$, then $b(z)=b(x)=m$ and hence $z\in P_m$.

For a nonempty directed $D\subseteq P$, write $s=\bigvee D$. Every $d\in D$ satisfies $d\le s$, so $b(d)=b(s)$. Consequently, if $s\in P_m$, then $D\subseteq P_m$; this proves Scott openness. If instead $D\subseteq P_m$, choose $d\in D$ and observe that $m\le d\le s$, whence $s\in P_m$. Together with lowerness this proves Scott closedness and the sub-dcpo assertion in~(2). Finally,~(1) shows that a unique minimal element lies below every point; the converse is immediate. This proves~(3).
\end{proof}

The next proposition extends the finite-support argument of \cite[Section~5]{XuJi} by allowing finitely many factors without least elements.

\Needspace{12\baselineskip}
\begin{proposition}\label{prop:finite-support-transfer}
Let $(P_i)_{i\in I}$ be a family of nonempty dcpos such that
\[
 \Sc\!\left(\prod_{i\in F}P_i\right)
 =\prod_{i\in F}\Sc(P_i)
 ~(F\subseteq I\text{ finite}).
\]
If $J=\{i\in I:P_i\text{ has no least element}\}$ is finite, then
\[
 \Sc\!\left(\prod_{i\in I}P_i\right)
 =\prod_{i\in I}\Sc(P_i).
\]
\end{proposition}
\begin{proof}
The case $I=\emptyset$ is immediate. Put $P=\prod_{i\in I}P_i$ and, for $i\notin J$, let $\bot_i$ be the least element of $P_i$. For $x\in P$ and a finite set $F$ with $J\subseteq F\subseteq I$, define
\begin{equation}\label{eq:finite-support-with-exceptions}
 x_i^F=
 \begin{cases}
  x_i,&i\in F,\\
  \bot_i,&i\notin F.
 \end{cases}
\end{equation}
The finite subsets $F$ containing $J$ form a directed set under inclusion. The corresponding points $x^F$ are directed, with supremum $x$: every coordinate is retained once $F$ contains its index.

Write $P_F=\prod_{i\in F}P_i$, let $\pi_F:P\to P_F$ be the projection, and let $e_F:P_F\to P$ insert $\bot_i$ in every coordinate outside $F$. Both maps preserve directed suprema. For every $W\in\Sc(P)$,
\begin{equation}\label{eq:finite-cylinder-with-exceptions}
 W=\bigcup_{\substack{F\subseteq I\text{ finite}\\J\subseteq F}}
       \pi_F^{-1}\bigl(e_F^{-1}(W)\bigr).
\end{equation}
Indeed, $x\in W$ implies $x^F\in W$ for some such $F$ by Scott openness, and membership in a set on the right implies $x^F\in W$ and $x^F\le x$, hence $x\in W$. The set $e_F^{-1}(W)$ is Scott open in $P_F$ and is therefore open in its finite product topology by hypothesis. Each cylinder on the right of~\eqref{eq:finite-cylinder-with-exceptions} is consequently open in $\prod_i\Sc(P_i)$. This proves one inclusion; Scott continuity of the coordinate projections gives the reverse inclusion.
\end{proof}

\Needspace{13\baselineskip}
\begin{theorem}\label{thm:arbitrary-L-products}
Let $(P_i)_{i\in I}$ be any family of nonempty countable meet-continuous domains, each of which is an $L$-dcpo. Then
\begin{equation}\label{eq:arbitrary-L-product}
 \Sc\!\left(\prod_{i\in I}P_i\right)
 =\prod_{i\in I}\Sc(P_i)
\end{equation}
holds if and only if the set
\[
 I_0=\{i\in I:P_i\text{ has no least element}\}
\]
is finite. Neither $I$ nor the whole product is required to be countable. The necessity of the finiteness condition holds for arbitrary nonempty $L$-dcpos, without countability or meet-continuity.
\end{theorem}
\begin{proof}
Suppose first that $I_0$ is finite. Theorem~\ref{thm:L-dcpo-products} gives the required equality on every finite subproduct, so Proposition~\ref{prop:finite-support-transfer} proves~\eqref{eq:arbitrary-L-product}.

For the converse, only the $L$-dcpo hypothesis is used. Suppose $I_0$ is infinite. For each $i\in I$, choose $m_i\in\operatorname{Min}(P_i)$ and put $m=(m_i)_{i\in I}$. Set
\begin{equation}\label{eq:minimal-component-product}
 W=\up m=\prod_{i\in I}\up m_i.
\end{equation}
This is an upper subset of $P=\prod_iP_i$. If a nonempty directed set $D\subseteq P$ has supremum $s\in W$, then for every $d\in D$ and $i\in I$,
\[
 b_i(d_i)=b_i(s_i)=m_i,
\]
by Lemma~\ref{lem:minimal-component-decomposition}, since $d_i\le s_i$ and $m_i\le s_i$. Hence $m_i\le d_i$ for every $i$, so in fact $D\subseteq W$. Therefore $W\in\Sc(P)$.

However, $W$ is not open in the ordinary product topology. A basic product neighbourhood $N$ of $m$ restricts only finitely many coordinates, say those in $F$. Choose $j\in I_0\setminus F$. By Lemma~\ref{lem:minimal-component-decomposition}\textup{(3)}, the factor $P_j$ has a minimal element $n_j\ne m_j$. Replace just the $j$-th coordinate of $m$ by $n_j$, obtaining a point $z\in N$. Distinct minimal elements are incomparable, so $z\notin W$. Thus no basic product neighbourhood of $m$ is contained in $W$, which contradicts~\eqref{eq:arbitrary-L-product}.
\end{proof}

\Needspace{9\baselineskip}
\begin{corollary}\label{cor:arbitrary-powers}
For a nonempty countable meet-continuous $L$-dcpo $P$, the following are equivalent:
\begin{enumerate}[label=\textup{(\arabic*)}]
\item $P$ has a least element;
\item $\Sc(P^{\N})=\prod_{n\in\N}\Sc(P)$;
\item $\Sc(P^I)=\prod_{i\in I}\Sc(P)$ for every index set $I$.
\end{enumerate}
\end{corollary}
\begin{proof}
Theorem~\ref{thm:arbitrary-L-products} gives~(1)$\Rightarrow$(3), and~(3)$\Rightarrow$(2) is immediate. If $P$ has no least element, every factor in its countable power lacks one, so the same theorem excludes~(2).
\end{proof}

\begin{example}\label{ex:antichain-infinite-product}
Let $A=\{a,b\}$ carry the equality order. Every directed subset of $A$ is a singleton, so $A$ is meet-continuous, and every principal ideal is a singleton complete lattice. Thus $A$ is a finite meet-continuous $L$-dcpo. Its countable order power $A^{\N}$ is again an antichain, so its Scott topology is discrete. The ordinary product of the discrete topologies on the factors is not discrete, since a basic neighbourhood restricts only finitely many coordinates. Consequently,
\[
 \prod_{n\in\N}\Sc(A)\subsetneq\Sc(A^{\N}).
\]
This is the elementary obstruction observed in \cite[Section~5]{XuJi}. Theorem~\ref{thm:arbitrary-L-products} shows that, within the present class of factors, infinitely many factors without least elements account for every failure of the arbitrary-product equality.
\end{example}

The preceding theorem concerns equality of topologies, not sobriety of the factors or their product. In particular, it does not settle Question~\ref{quest:mc-dcpo-sober}. Its necessity argument uses local least elements and is not asserted for the broader class in Corollary~\ref{cor:bottom-free-local-products}.

\section{Further remarks}
Theorem~\ref{thm:L-dcpo-products} and Corollary~\ref{cor:bottom-free-local-products} answer Question~\ref{quest:mc-dcpo-product} under local completeness hypotheses. The missing step for a general countable meet-continuous domain is to obtain an adequate substitute for these local test families, or to construct a counterexample to the product equality. The present argument does not assert that its family of ideals is countable.

Theorem~\ref{thm:arbitrary-L-products} gives the exact additional requirement for arbitrary products within the class of countable meet-continuous $L$-dcpos: all but finitely many factors must have least elements. This criterion does not remove the local lattice hypothesis in the original finite-product question.

For Question~\ref{quest:mc-dcpo-sober}, Proposition~\ref{prop:bounded-complete-sober} gives a positive bounded-completeness case, while Corollaries~\ref{cor:conditional-sobriety} and~\ref{cor:finite-upper-bound-generators} give additional sufficient conditions within the local-completeness class. Assuming the square product equality, Proposition~\ref{prop:irreducible-upper-bound-sober} and Corollary~\ref{cor:sequential-upper-bound} reduce sobriety to the common-upper-bound lifting condition for irreducible Scott-closed sets. They do not establish that this condition follows from countability and meet-continuity.

In particular, sobriety for all countable meet-continuous $L$-dcpos is not claimed. Their principal ideals are sober by the complete-lattice theorem, but no unproved local-to-global sobriety principle is used. Example~\ref{ex:Ldcpo-not-strong} also rules out a proof that would first derive the strong $d$-space property for every member of this class. These distinctions leave both general questions available for further investigation without changing the hypotheses of the proved statements.


\begin{thebibliography}{99}
\raggedright
\bibitem{AbramskyJung}
S. Abramsky, A. Jung,
Domain theory,
in: S. Abramsky, D. Gabbay, T. Maibaum (Eds.),
\emph{Handbook of Logic in Computer Science}, vol.~3,
\emph{Semantic Structures}, Clarendon Press, 1994, pp.~1--168.

\bibitem{Gierz}
G. Gierz, K. H. Hofmann, K. Keimel, J. D. Lawson, M. Mislove, D. S. Scott,
\emph{Continuous Lattices and Domains},
Encyclopedia of Mathematics and its Applications, vol.~93,
Cambridge University Press, Cambridge, 2003.
\href{https://doi.org/10.1017/CBO9780511542725}{doi:10.1017/CBO9780511542725}.

\bibitem{JiaJungLi}
X. Jia, A. Jung, Q. Li,
A dichotomy result for locally compact sober dcpos,
\emph{Houston J. Math.} 45 (3) (2019), 935--951.


\bibitem{KouLiuLuo}
H. Kou, Y.-M. Liu, M.-K. Luo,
On meet-continuous dcpos,
in: G. Q. Zhang, J. D. Lawson, Y.-M. Liu, M.-K. Luo (Eds.),
\emph{Domain Theory, Logic and Computation},
Semantic Structures in Computation, vol.~3,
Springer, Dordrecht, 2003, pp.~117--135.
\href{https://doi.org/10.1007/978-94-017-1291-0_5}{doi:10.1007/978-94-017-1291-0\_5}.

\bibitem{LawsonXu}
J. Lawson, X. Xu,
Spaces determined by countably many locally compact subspaces,
\emph{Topology Appl.} 354 (2024), 108973.
\href{https://doi.org/10.1016/j.topol.2024.108973}{doi:10.1016/j.topol.2024.108973}.

\bibitem{LawsonXuLower}
J. Lawson, X. Xu,
$T_0$-spaces and the lower topology,
\emph{Math. Struct. Comput. Sci.} 34 (2024), 467--490.
\href{https://doi.org/10.1017/S0960129524000240}{doi:10.1017/S0960129524000240}.


\bibitem{LyuKou}
Z. Lyu, H. Kou,
The probabilistic powerdomain from a topological viewpoint,
\emph{Topology Appl.} 237 (2018), 26--36.
\href{https://doi.org/10.1016/j.topol.2018.01.008}{doi:10.1016/j.topol.2018.01.008}.


\bibitem{Miao}
H. Miao, X. Xi, Q. Li, D. Zhao,
Not every countable complete distributive lattice is sober,
\emph{Math. Struct. Comput. Sci.} 33 (9) (2023), 809--831.
\href{https://doi.org/10.1017/S0960129523000269}{doi:10.1017/S0960129523000269}.

\bibitem{XuJi}
X. Xu, W. Ji,
Every countable meet-continuous lattice is Scott sober,
preprint, 2026. arXiv:2609.19612v1.

\end{thebibliography}
\end{document}